\documentclass[12pt,a4paper,oneside,reqno,final]{amsart}

\usepackage{amsthm, amsmath, amssymb,amscd, amsfonts, mathtools}
\usepackage[T1]{fontenc}
\usepackage[utf8]{inputenc}
\usepackage[english]{babel}
\usepackage{mathrsfs}
\usepackage{verbatim}
\usepackage{color}
\usepackage{float}                                                  % For figures
\usepackage{hyperref}
\hypersetup{
    pdftitle={Ideals in etale groupoid C*-algebras},
      colorlinks=true,
    linkcolor={blue!50!black},
      citecolor={red!50!black},
      urlcolor={blue!80!black},
      linktocpage=true                                                      % only page numbers are linked in toc
  }
\usepackage{bookmark}
\usepackage[noabbrev,capitalise,nameinlink]{cleveref}
\usepackage{enumitem}                                               % For enumeration
  \setenumerate{listparindent=\parindent}
  \setlist[enumerate]{font=\normalfont}

\usepackage{tikz}                                                   % tikz
\usetikzlibrary{cd, matrix, arrows, decorations.markings,quotes}

\newcommand{\N}{\mathbb{N}}
\newcommand{\Z}{\mathbb{Z}}

\newcommand{\T}{\mathcal{T}}
\newcommand{\I}{\mathcal{I}}
\newcommand{\Iso}{\textup{Iso}}                                     % Isotropy

\newcommand{\TT}{\mathbb{T}}

\DeclareMathOperator{\orb}{orb}
\DeclareMathOperator{\supp}{supp}

\let\tilde\widetilde{}
\let\leq\leqslant{}
{}
\let\subset\subseteq{}
{}

\newcommand{\eff}{\operatorname{eff}}
\newcommand{\ob}{\operatorname{ob}}

\newtheorem{lemma}{Lemma}%[section]
\newtheorem{corollary}[lemma]{Corollary}
\newtheorem{theorem}[lemma]{Theorem}
\newtheorem{proposition}[lemma]{Proposition}

\theoremstyle{definition}
\newtheorem{definition}[lemma]{Definition}
\newtheorem{example}[lemma]{Example}
\newtheorem{examples}[lemma]{Examples}
\newtheorem{remark}[lemma]{Remark}
\newtheorem*{theorem*}{Theorem}

\numberwithin{lemma}{section}
\numberwithin{equation}{section}                                    % Equation numbering

\title[Ideals in \'etale groupoid $C^*$-algebras]{Some results regarding the ideal structure of $C^*$-algebras of \'etale groupoids}
\author[K.A.~Brix]{Kevin Aguyar Brix}
\address[K.A.~Brix]{School of Mathematics and Statistics \\ University of Glasgow \\ University Place, G12 8QQ, United Kingdom}
\email[K.A.~Brix]{kabrix.math@fastmail.com}
\author[T.M.~Carlsen]{Toke Meier Carlsen}
\address[T.M.~Carlsen]{K\o{}ge, Denmark}
\email[T.M.~Carlsen]{toke.carlsen@gmail.com}
\author[A.~Sims]{Aidan Sims}
\address[A.~Sims]{School of Mathematics and Applied Statistics\\ University of Wollongong\\ Wollongong Australia 2522}
\email[A.~Sims]{asims@uow.edu.au}
\makeatletter
\@namedef{subjclassname@2020}{%
  \textup{2020} Mathematics Subject Classification}
\makeatother


\thanks{This research was supported by Australian Research Council grant DP200100155.
KAB acknowledges the support of the Carlsberg Foundation via an Internationalisation Fellowship and the support of the Independent Research Fund Denmark (Case number 1025-00004B)}

\begin{document}
\begin{abstract}
We prove a sandwiching lemma for inner-exact locally compact Hausdorff
\'etale groupoids. Our lemma says that every ideal of the reduced
$C^*$-algebra of such a groupoid is sandwiched between the ideals
associated to two uniquely defined open invariant subsets of the unit
space. We obtain a bijection between ideals of the reduced $C^*$-algebra,
and triples consisting of two nested open invariant sets and an ideal in
the $C^*$-algebra of the subquotient they determine that has trivial
intersection with the diagonal subalgebra and full support. We then
introduce a generalisation to groupoids of Ara and Lolk's relative strong
topological freeness condition for partial actions, and prove that the
reduced $C^*$-algebras of inner-exact locally compact Hausdorff \'etale
groupoids satisfying this condition admit an obstruction ideal in Ara and
Lolk's sense.
\end{abstract}
\maketitle

\section{Introduction}

The purpose of this paper is to investigate the ideal structure of the
reduced $C^*$-algebras of locally compact Hausdorff \'etale groupoids. This
very broad class of $C^*$-algebras contains all reduced crossed products of
commutative $C^*$-algebras by discrete groups. It also includes graph
$C^*$-algebras \cite{KPRR}, higher-rank graph $C^*$-algebras
\cite{KumjianPask2000}, the models described by Spielberg \cite{Spielberg}
and Katsura \cite{KatsuraModels} for Kirchberg algebras, the stable and
unstable Ruelle algebras of Smale spaces (up to Morita equivalence), and many
self-similar action $C^*$-algebras \cite{ExelPardo}.

Among the more natural invariants of a $C^*$-algebra, but also among the most
difficult to compute, is its lattice of ideals. In the situation of \'etale
groupoid $C^*$-algebras, definitive theorems are available for $C^*$-algebras
of amenable groupoids that are essentially principal in the sense of Renault
\cite{Renault1991}, graph $C^*$-algebras \cite{anHuef-Raeburn1997,
Hong-Szymanski}, and for $C^*$-algebras of a single local homeomorphisms
\cite{KatsuraIdeals}, but few other truly general results about ideal
structure of groupoid $C^*$-algebras are available.

The analysis of ideals in \'etale groupoid $C^*$-algebras typically has two
components. The first is concerned with what we call here dynamical ideals,
and is well understood. The continuous functions on the unit space $G^{(0)}$
of an \'etale groupoid $G$ embed as a $\sigma$-unital subalgebra  $D$ of the
groupoid $C^*$-algebra. So each ideal $I$ of $C^*(G)$ yields an ideal $I \cap
D$ of $D$ and hence an open subset $U_I$ of $G^{(0)}$ on which it is
supported. This $U_I$ is invariant in the sense that if $s(\gamma) \in U_I$
then $r(\gamma) \in U_I$. If $I$ is generated as an ideal by $I \cap D$, we
call it a \emph{dynamical ideal}. The assignment $I \mapsto U_I$ is a lattice
isomorphism between dynamical ideals of $C^*(G)$ and open invariant sets of
$G^{(0)}$, giving a complete description of the dynamical ideals. In
particular, by identifying the \emph{essentially principle} (now sometimes
referred to instead as \emph{strongly effective}) and amenable groupoids for
which every ideal of $C^*_r(G)$ is dynamical, Renault gives a complete
description of the ideal structure for this class of groupoid $C^*$-algebras
\cite{Renault1991}.

The second component of the analysis is more complicated. It amounts to
understanding the collection of all possible ideals that have fixed
intersection with $D$. For full $C^*$-algebras, this is, in general,
hopelessly intractable: there is a zoo of ideals contained in the kernel of
the regular representation, which has trivial intersection with $D$, alone.
So we are led to restrict our attention to reduced $C^*$-algebras. Another
problem arises almost immediately: given an open invariant set $U$, the
restriction map $f \mapsto f|_{G \setminus G|_U}$ on $C_c(G)$ extends to a
homomorphism $C^*_r(G) \to C^*_r(G \setminus G|_U)$ whose kernel contains the
dynamical ideal $I_U$ associated to $U$. In the setting of full
$C^*$-algebras, this containment is an equality, but for reduced
$C^*$-algebras it need not be: as Willet's example \cite{Willett} shows, the
quotient $C^*_r(G)/I_U$ can coincide with the full $C^*$-algebra $C^*(G
\setminus G|_U)$, and we encounter the same zoo of ideals as before. So we
restrict attention further to groupoids that are inner-exact in the sense
that $C^*_r(G)/I_U \cong C^*_r(G \setminus G|_U)$ for every open invariant set
$U$. Lest this seem overly restrictive, note that this includes all amenable
\'etale groupoids $G$, and therefore all nuclear \'etale groupoid
$C^*$-algebras \cite{A-DR}.

In this setting, existing results rely, explicitly or otherwise, on a kind of
sandwiching lemma. This technique was developed by an Huef and Raeburn
\cite{anHuef-Raeburn1997} to analyse Cuntz--Krieger algebras. Here the
dynamical ideals are better known as \emph{gauge-invariant ideals} (see
\cref{prop:gauge-inv ideals}). To understand the ideals of a Cuntz--Krieger
algebra, an Huef and Raeburn concentrate on primitive ideals and demonstrate
that for each primitive ideal $I$ there are a unique smallest gauge-invariant
ideal $K$ containing $I$ and largest gauge-invariant ideal $J$ contained in
$I$. They then analyse the quotient $K/J$, which is itself Morita equivalent
to a graph algebra---but of a graph consisting of just one vertex and one
edge. The $C^*$-algebra of this graph is $C(\TT)$, so its ideal structure is
well understood, and their analysis proceeds from there. A similar idea was
used in \cite{Hong-Szymanski}, and again in \cite{KatsuraIdeals} to
understand ideal structure first for graph $C^*$-algebras and then for
topological-graph algebras, viewed as $C^*$-algebras associated to singly
generated irreversible dynamics.

Another instance of the same idea appears in Ara and Lolk's very interesting
work on partial actions \cite{Ara-Lolk2018}. They identify a relative strong
topological freeness condition that generalises Renault's topologically
principle condition in the setting of transformation groupoids for partial
actions. They show that relative strong topological freeness guarantees the
existence of an \emph{obstruction ideal}: a smallest dynamical ideal of
$C^*_r(G)$ that contains every ideal with trivial intersection with $D$. This
can again be regarded as a kind of sandwiching result, but with the
quantifiers switched: there exists a pair of dynamical ideals, namely the
zero ideal and the obstruction ideal, that sandwich every ideal that has
trivial intersection with $D$. One of our motivations in writing this paper
is that, because this particular aspect of Ara and Lolk's paper appears as a
technical step along the way to their main objective, it is in danger of
receiving less attention than we think it deserves, and we want to advertise
the idea more broadly.

In this paper, we take up the idea of the sandwiching lemma and of Ara and
Lolk's relative strong topological freeness condition and obstruction ideal.
We first establish a general sandwiching lemma for groupoid $C^*$-algebras
(\cref{lem:sandwich}): given any inner-exact locally compact Hausdorff
\'etale groupoid $G$, and any ideal $I$ of $C^*_r(G)$, there are a unique
smallest dynamical ideal $K$ containing $I$ and largest dynamical ideal
contained in $I$. As a result the ideals of $C^*_r(G)$ are parameterised by
triples $(U, V, J)$ consisting of open invariant sets $U \subseteq V
\subseteq G^{(0)}$, and an ideal $J$ of $C^*_r(G|_V \setminus G|_U)$ that has
trivial intersection with $D$ and vanishes nowhere on $G|_V \setminus G|_U$
(\cref{thm:all ideals}).

We then adapt Ara and Lolk's notions of topological freeness and strong
topological freeness at a point (see also Renault's notion of discretely
trivial isotropy \cite{Renault1991}), and of relative strong topological
freeness, from their setting of partial actions of groups to the setting of
\'etale groupoids. We identify a condition on \'etale groupoids, which we
phrase as being \emph{jointly effective where they are effective}, that
ensures that $C^*_r(G)$ admits an obstruction ideal in the sense of Ara and
Lolk (see \cref{thm:obstruction-ideal} and~\cref{cor:minimal}). We also show
that this obstruction ideal is minimal in the strong sense that there exists
an ideal that has trivial intersection with $D$ and whose support exhausts
the support of the obstruction ideal. We show that any groupoid whose
isotropy groups are all either trivial or infinite cyclic is jointly
effective where it is effective. This includes all graph groupoids and
groupoids arising from single local homeomorphisms. In our companion paper \cite{BCS23},
we show how to use our results to give a complete description of the ideal
structure of a large class of Deaconu--Renault groupoid $C^*$-algebras,
including those considered in \cite{anHuef-Raeburn1997, Hong-Szymanski,
KatsuraIdeals} and all $C^*$-algebras of rank-2 graphs.

\smallskip

The paper is arranged as follows. We introduce the background we need in
~\cref{sec:prelims}. In~\cref{sec:sandwich} we prove our
sandwiching lemma and explore its consequences. In
~\cref{sec:obstruction ideal} we introduce the notions of a groupoid
being effective at a unit, strongly effective at a unit, and being strongly
effective where it is effective. We then prove that such groupoids admit an
obstruction ideal, and discuss some examples. Finally in
~\cref{sec:examples}, we present examples of groupoids that are jointly
effective where they are effective, and describe the support of the
obstruction ideal; in particular, we devote~\cref{sec:AraLolk} to
showing exactly how our work in~\cref{sec:obstruction ideal}
generalises the ideas of Ara and Lolk.

\section{Preliminaries}\label{sec:prelims}

\subsection{Hausdorff \'etale groupoids}
We will always be working with topological groupoids that are locally
compact, Hausdorff, and \'etale, and we shall adopt most of the notation and
terminology from~\cite{SimsCRM} (see also~\cite{Renault}).

We consider the unit space $G^{(0)}$ as a locally compact Hausdorff subspace
of $G$, and we denote range and source maps by $r,s\colon G\to G^{(0)}$. A
\emph{bisection} is a subset $B$ of $G$ such that both $r$ and $s$ restrict
to injective maps on $B$. That $G$ is Hausdorff means that the unit space
$G^{(0)}$ is a closed subset of $G$, and that $G$ is \'etale (in the sense
that the range and source maps are local homeomorphisms) implies that
$G^{(0)}$ is also open, that $G$ has a basis consisting of open bisections,
and that the range and source fibres over a unit $x\in G^{(0)}$ given by $G^x
= \{\gamma\in G : r(\gamma) = x\}$ and $G_x = \{ \gamma\in G : s(\gamma) =
x\}$, respectively, are discrete in the relative topology. In particular, the
isotropy group over a unit $x\in G^{(0)}$ given as the intersection
\[
  \I(G)_x = G^x\cap G_x = \{ \gamma\in G : r(\gamma) = x = s(\gamma) \}
\]
is a discrete subgroup of $G$.
A unit $x$ is said to have trivial isotropy if $\I(G)_x = \{x\}$.
The isotropy subgroupoid is then the group bundle
\[
  \I(G) = \bigsqcup_{x\in G^{(0)}} \I(G)_x = \{\gamma\in G : r(\gamma) = s(\gamma)\}.
\]
We write $\I^\circ(G)$ for the topological interior of the isotropy of $G$. A
Hausdorff groupoid $G$ is said to be \emph{effective} if $\I^\circ(G) =
G^{(0)}$, that is, the interior of the isotropy subgroupoid with the subspace
topology coincides with the unit space. When $G$ is second-countable this
coincides (using a Baire category argument) with the notion of $G$ being
\emph{topologically principal} in the sense that $G^{(0)}$ has a dense set of
points with trivial isotropy.

\subsection{Reduced groupoid C*-algebra}
We will be working with the reduced groupoid $C^*$-algebras of locally
compact Hausdorff \'etale groupoids. We follow the exposition
of~\cite{SimsCRM}.

The convolution algebra $C_c(G)$ of a locally compact Hausdorff \'etale
groupoid $G$ is the set of compactly supported and complex-valued functions
on $G$ equipped with the convolution product
\[
  f\ast g (\gamma) = \sum_{\alpha \in G^{r(\gamma)}} f(\alpha) g(\alpha^{-1} \gamma)
\]
for all $f,g\in C_c(G)$ and $\gamma\in G$, and the involution $f^*(\gamma) =
\overline{f(\gamma^{-1})}$ for all $f\in C_c(G)$ and $\gamma\in G$. Each unit
$x\in G^{(0)}$ determines a regular representation $\pi_x \colon C_c(G) \to
B(\ell^2(G_x))$ given by
\[
  \pi_x(f) \delta_\gamma = \sum_{\alpha \in G_{r(\gamma)}} f(\alpha) \delta_{\alpha \gamma},
\]
for all $f\in C_c(G)$ and $\gamma\in G_x$. The \emph{reduced groupoid
$C^*$-algebra} $C^*_r(G)$ of $G$ is the completion of $\bigoplus_{x\in
G^{(0)}} \pi_x(C_c(G))$ in $\bigoplus_{x\in G^{(0)}} B(\ell^2(G_x))$. Since
the unit space $G^{(0)}$ is both open and closed in $G$, the commutative
algebra $C_0(G^{(0)})$ sits naturally as a subalgebra of $C^*_r(G)$, and we
refer to $C_0(G)$ as the diagonal subalgebra. Note that $C_0(G^{(0)})$ need
not be a $C^*$-diagonal (in the sense of Kumjian~\cite{Kumjian1986}) nor a
Cartan subalgebra (in the sense of Renault~\cite{Renault2008}).

Renault~\cite[Proposition~II.4.2]{Renault} shows (Renault makes the standing
assumption that the groupoids considered there are second-countable, but that
assumption is not needed for the following) that any element in the reduced
groupoid $C^*$-algebra may be thought of as a function on the groupoid. More
precisely, there exists a linear and norm-decreasing map $j\colon C^*_r(G)
\to C_0(G)$ given by
\[
  j(a)(\gamma) = \big( \pi_{s(\gamma)}(a) \delta_{s(\gamma)} \mid \delta_\gamma  \big)
\]
for all $a\in C^*_r(G)$ and $\gamma\in G$, and $j$ is the identity on
$C_c(G)$. The reduced groupoid $C^*$-algebra admits a faithful conditional
expectation $E \colon C^*_r(G) \to C_0(G^{(0)})$ onto the diagonal given by
restriction of functions in the sense that $j(E(a)) = j(a)|_{G^{(0)}}$ for
all $a \in C^*_r(G)$ \cite[Proposition 10.2.6]{SimsCRM}.
Renault shows that for $a,b \in C^*_r(G)$, the convolution formula for
$j(a)*j(b)$ is a convergent series that converges to $j(a*b)$.

A subset $U$ of $G^{(0)}$ is $G$-invariant (or simply invariant) if $r(GU)
\subset U$, and the reduction of $G$ to $U$ is $G|_U = \{ \gamma\in G :
r(\gamma), s(\gamma) \in U\}$. If $U$ is an open and invariant subset of
$G^{(0)}$, then $G|_U$ is an open subgroupoid of $G$ (and hence locally
compact, Hausdorff, and \'etale), and the inclusion $i_U \colon C_c(G|_U)\to
C_c(G)$ extends to an injective $^*$-homomorphism $i_U\colon C^*_r(G|_U) \to
C^*_r(G)$. We let $I_U \coloneqq i_U(C^*_r(G|_U))$ be the image of $i_U$ in
$C^*_r(G)$. This is an ideal with the property that $I_U \cap C_0(G^{(0)}) =
C_0(U)$ and $I_U$ is generated as an ideal by $C_0(U)$. We shall refer to
such ideals as \emph{dynamical ideals} (see~\cref{def:dynamical-ideal}).

The complement $G^{(0)} \setminus U$ is a closed invariant set of units, and
there is a $^*$-homomorphism $\pi_U \colon C^*_r(G) \to C^*_r(G|_{G^{(0)}
\setminus U})$ determined by $\pi_U(f) = f|_{G|_{G^{(0)}\setminus U}}$ for
all $f\in C_c(G)$.

Given an ideal $I$ in $C^*_r(G)$, we write $\supp(I) \coloneqq \{\gamma \in G
: j(I)(\gamma) \not= \{0\}\}$. So $\supp(I_U) = G|_U$ for every open
invariant $U \subseteq G^{(0)}$ (for completeness, we prove this in
Proposition~\ref{prop:dynamical ideals}.

\begin{lemma}\label{lem:support invariant}
Let $G$ be a locally compact Hausdorff \'etale groupoid. Suppose that $I$ is
an ideal of $C^*_r(G)$. Then $\supp(I)$ is invariant under multiplication and
inversion in $G$.
\end{lemma}
\begin{proof}
Fix $\gamma \in \supp(I)$ and take $\alpha \in G_{r(\gamma)}$ and $\beta \in
G_{s(\gamma)}$. Fix $a \in I$ such that $j(a)\gamma \not= 0$. Take open
bisections $B$ and $C$ containing $\alpha$ and $\beta$, respectively, and
take $h \in C_c(B)$ and $k \in C_c(C)$ with $h(\alpha) = k(\beta) = 1$. Then
$j(hak)(\alpha\gamma\beta) = j(a)(\gamma) \not= 0$, so $\alpha\gamma\beta \in
\supp(I)$. Putting $\beta = s(\gamma)$ gives invariance under left
multiplication, putting $\alpha = r(\gamma)$ gives invariance under right
multiplication, and putting $\alpha = \beta = \gamma^{-1}$ gives invariance
under inversion.
\end{proof}

The groupoid $G$ is \emph{inner-exact} if, for every open invariant subset $U
\subseteq G^{(0)}$, the resulting sequence
\begin{equation}\label{eq:inner-exact}
  0 \to C^*_r(G|_U) \to C^*_r(G) \to C^*_r(G|_{G^{(0)}\setminus U}) \to 0
\end{equation}
is exact, (see~\cite[Definition~3.7]{A-Delaroche} and
also~\cite[Definition~3.5]{Bonicke-Li2017}). Any amenable groupoid is inner
exact. Combining Proposition 4.23 and Theorem 7.10 in \cite{Anatharaman-Delaroche2021},
we also see that the (partial) crossed product groupoid of an exact group acting
(partially) on a second-countable locally compact Hausdorff space is inner-exact. Willett's
example of a nonamenable groupoid whose full and reduced $C^*$-algebras
coincide is not inner-exact~\cite{Willett}.

\begin{remark}
The empty set satisfies the axioms defining a locally compact Hausdorff \'etale groupoid. 
By convention, we take the $C^*$-algebra of the empty groupoid to be the zero $C^*$-algebra; 
in particular~\eqref{eq:inner-exact} collapses to the exact sequence $0 \to C^*_r(G) \to C^*_r(G) \to 0$ if $U \in \{\varnothing, G^{(0)}\}$. 
We thank the referee for pressing us on this point.
\end{remark}

\section{A sandwiching lemma for Hausdorff \'etale
groupoids}\label{sec:sandwich}

The characterisations of the primitive-ideal spaces of graph $C^*$-algebras
of~\cite{Hong-Szymanski} and~\cite{anHuef-Raeburn1997} were founded on the
``sandwiching lemmas''~\cite[Lemma~2.6]{Hong-Szymanski}
and~\cite[Lemma~4.5]{anHuef-Raeburn1997} that show that every primitive ideal
is sandwiched between a pair of uniquely determined gauge-invariant ideals.
Here we observe that a similar sandwiching lemma holds for ideals of reduced
Hausdorff \'etale groupoid $C^*$-algebras.

\begin{definition} \label{def:dynamical-ideal}
We say that an ideal $I$ in a reduced groupoid $C^*$-algebra $C^*_r(G)$ is
\emph{dynamical} if it is generated as an ideal by its intersection with the
diagonal subalgebra $C_0(G^{(0)})$. Equivalently, $I$ is dynamical if it is
of the form $I_U$ for an open invariant subset $U$ of $G^{(0)}$. We say that
$I$ is \emph{purely non-dynamical} if $I\cap C_0(G^{(0)}) = \{0\}$.
\end{definition}

\begin{remark}
According to \cref{def:dynamical-ideal}, the trivial ideal $\{0\}$ is the unique ideal of $C^*_r(G)$ that is both a dynamical ideal and a purely non-dynamical ideal. 
Though linguistically unsatisfactory, this convention simplifies the statements of our key results: 
in \cref{prop:dynamical ideals} treating $\{0\}$ as a dynamical ideal avoids treating the open invariant set $\varnothing$ as a special case; 
but later in \cref{thm:all ideals}, treating $\{0\}$ as a purely non-dynamical ideal avoids treating dynamical ideals as a special case---see \cref{rmk:Theta->dynamical}.
\end{remark}

In the context of Deaconu--Renault groupoids, the dynamical ideals are precisely the usual gauge-invariant ideals---see \cref{prop:gauge-inv ideals}.

\begin{proposition}\label{prop:dynamical ideals}	
  Let $G$ be a locally compact Hausdorff \'etale groupoid.
  The map $U \mapsto I_U$ is a lattice isomorphism from the lattice of open invariant subsets of $X$ to the lattice of
  dynamical ideals of $C^*_r(G)$.
  For each open invariant $U \subseteq G^{(0)}$, we have $I_U \cap C_0(G^{(0)}) = C_0(U)$, and $\supp(I_U) = G|_U$.
\end{proposition}

\begin{proof}
  The map $U \mapsto I_U$ is always an injection \cite[Theorem~10.3.3]{SimsCRM},
  and surjectivity follows from the definition of dynamical ideals.
  Proposition~10.3.2 of~\cite{SimsCRM} shows that $I_U$ is the closure of $C_c(G|_U) \subseteq C_c(G)$.
  In particular, $\supp(I_U) \subseteq G|_U$, and $I_U \cap C_0(G^{(0)}) \subseteq C_0(U)$ by continuity of $j(a)$ for
  each $a\in C^*_r(G)$.
  The reverse containments hold because if $\gamma\in G|_U$, then there is $f\in C_c(G|_U) \subset I_U$ such that $f(\gamma)\neq 0$,
  so $\gamma\in \supp(I_U)$,
  and $C_c(U)$ is contained in $C_c(G|_U)$, so $C_0(U)$ is contained in $I_U$.
\end{proof}

Since lattice isomorphisms preserve least upper bounds and greatest lower bounds, it follows from
\cref{prop:dynamical ideals} that, for example, $I_U \cap I_V = I_{U \cap V}$ and $I_U + I_V = I_{U \cup V}$
for all open invariant $U$ and $V$.

We now state our sandwiching lemma.

\begin{lemma}[The sandwiching lemma]\label{lem:sandwich}
  Let $G$ be a locally compact Hausdorff \'etale groupoid that is inner-exact and let $I$ be an ideal of $C^*_r(G)$.
  Consider the open and invariant subsets
  \[
    U = \{x \in G^{(0)} : f(x) \not= 0\text{ for some } f \in I \cap C_0(G^{(0)})\}
  \]
  and
  \[
    V = \{x \in G^{(0)} : j(a)(x) \not= 0\text{ for some } a \in I\}.
  \]
  Then $I_U$ is the largest dynamical ideal of $C^*_r(G)$ contained in $I$ and $I_V$ is the smallest dynamical ideal of
  $C^*_r(G)$ containing $I$.
\end{lemma}

\begin{proof}
  The set $U$ is open because every $f\in C_0(G^{(0)})$ is continuous.
  In order to see that $U$ is invariant, take $x\in U$ and fix $\gamma\in G_x$.
  We will show that $r(\gamma) \in U$.
  Since $x \in U$ there exists $f\in I\cap C_0(G^{(0)})$ such that $f(x) \neq 0$.  
  Let $B$ be an open bisection containing $\gamma$ and fix $h\in C_c(B)$ such that $h(\gamma) = 1$.
  Then $hfh^* \in I \cap C_c(r(B)) \subseteq I \cap C_0(G^{(0)})$.
  Moreover, $hfh^*(r(\gamma)) = h(\gamma) f(x) h^*(\gamma^{-1}) = f(x) \neq 0$,
  so we conclude that $r(\gamma)\in U$.

  For each $x \in U$, choose $f_x \in I \cap C_0(G^{(0)})$ such that $f_x(x) \not= 0$.
  Then $\{f_x : x \in U\}$ generates $C_0(U)$ as an ideal of $C_0(G^{(0)})$, and it is contained in $I$.
  Hence $I_U \subseteq I$.
  Suppose that $U'$ is an open subset of $G^{(0)}$ strictly containing $U$ and fix $x \in U' \setminus U$ and $f \in
  C_c(U')$ with $f(x) \not= 0$.
  Then $f \in I_{U'}$ but $f \not\in I_U$ by definition of $I_U$.
  In particular, $f\notin I$, so $I_{U'} \not\subseteq I$.
  This proves that $I_U$ is the largest dynamical ideal contained in $I$.

  The set $V$ is open because $j(a)$ is continuous for every $a\in C^*_r(G)$.
  We claim that $V = s(\supp(I))$.
  That $V \subseteq s(\supp(I))$ is obvious. For the reverse inclusion, suppose that
  $a \in I$ and $j(a)(\gamma) \neq 0$. For any open bisection $B$ containing $\gamma^{-1}$ and any $f \in
  C_c(B)$ satisfying $f(\gamma^{-1}) = 1$,
  we have $j(fa)(s(\gamma)) = j(a)(\gamma) \neq 0$.
  Since $j(a)(\gamma) = \overline{j(a^*)(\gamma^{-1})}$, we have $r(\gamma) \in V$ if and only if $s(\gamma)
  \in V$, so $V$ is invariant.

  We now show that $I\subset I_V$.
  Let $E\colon C^*_r(G) \to C_0(G^{(0)})$ be the faithful conditional expectation onto the diagonal
  and observe that $E(I) \subset C_0(V)$.
  Since $G$ is inner-exact, it follows from~\cite[Lemma 3.6]{Bonicke-Li2017} that $I$ is contained in the ideal in
  $C^*_r(G)$ generated by $E(I)$,
  so we find that $I \subset I_V$ as wanted. To see that $V$ is minimal with
  this property, suppose that $V' \subsetneq V$ is an open invariant set. By definition of $V$ there exists $x \in V \setminus
  V'$ and $a \in I$ such that $\jmath(a)(x) \not= 0$.
  Hence $\supp(I) \not\subset \supp(I_{V'})$, so $I \not\subseteq I_{V'}$.
\end{proof}

\begin{remark}%\label{rmk:purely nondynamical U empty}
If the ideal $I$ in \cref{lem:sandwich} is a purely non-dynamical ideal of $C^*_r(G)$, then $U$ is empty, and then $I_U = \{0\}$; if $I$ is a dynamical ideal, then $V = U$ and $I_U = I$.
\end{remark}

Consider a pair of nested open invariant subsets $U\subset V \subset
G^{(0)}$. Recall that we obtain $C^*$-homomorphisms $i_V : C_r^*(G|_V) \to
C_r^*(G)$ and $\iota_{V\setminus U} : C^*_r(G|_{V\setminus U}) \to
C^*_r(G|_{G^{(0)}\setminus U})$ extending the canonical inclusion of algebras
of compactly supported functions. For these maps, the diagram
\[
\begin{tikzcd}
  C^*_r(G|_U) \arrow[r, "\iota_U"] & C^*_r(G) \arrow[r, "\pi_U"]  & C^*_r(G|_{G^{(0)}\setminus U})\\
  C^*_r(G|_U)\arrow[r, "\iota_U^V"'] & C^*_r(G|_V) \arrow[r, "\pi_U^V"'] \arrow[u, "\iota_V"] & C^*_r(G|_{V\setminus
  U}) \arrow[u, "\iota_{V\setminus U}"]
\end{tikzcd}
\]
commutes.

\begin{lemma}
  Let $G$ be a locally compact Hausdorff \'etale groupoid that is inner-exact.
  Let $I$ be an ideal of $C^*_r(G)$ and let $U$ and $V$ be the open invariant sets of~\cref{lem:sandwich}.
  Then $J \coloneqq \pi_U^V(\iota_V^{-1}(I))$ is an ideal in $C^*_r(G|_{V\setminus U})$ that is purely non-dynamical
  and has full support.
\end{lemma}

\begin{proof}
  It is clear that $J$ is an ideal of $C^*_r(G|_{V\setminus U})$.
  In order to see that $J$ is purely non-dynamical, take $f\in J\cap C_c(V\setminus U)$.
  Pick $\tilde{f}\in \pi_U^{-1}(f)$ and note that $\tilde{f}\in C_0(V)$ extends $f$ (because $\pi_U^V$ implements
  restriction of functions).
  Then $\iota_V(\tilde{f})\in I$ by definition of $J$.
  If $x\in V\setminus U$, then $\iota_V(\tilde{f})(x) = 0$ by definition of $U$, so $f(x) = 0$.
  Hence $f=0$. So $J$ is purely non-dynamical.

  Next we show that $J$ has full support.
  Clearly, $\supp(J) \subset G|_{V\setminus U}$ (since $J$ is an ideal of $C^*_r(G|_{V \setminus U})$). We must
  prove the reverse inclusion.
  Fix $\gamma \in G$ with $s(\gamma) \in V \setminus U$.
  Since $V = s(\supp(I))$, there exists $a \in I$ such that $j(a)(\gamma) \not= 0$.
  The inclusion map $C^*_r(G|_V) \to I_V \subseteq C^*_r(G)$ extends the canonical inclusion $C_c(G_V) \to C_c(G)$,
  so it intertwines the maps $j^V\colon C^*_r(G|_V) \to C_0(G|_V)$ and $j\colon C^*_r(G) \to C_0(G)$.
  Therefore $j^V(\iota_V^{-1}(a))(\gamma) = j(a)(\gamma) \not= 0$, and we conclude that $\supp(J) = G|_{V\setminus U}$.
\end{proof}

Let $\T(G)$ be the collection of triples $(U,V,J)$ where $U \subset V\subset G^{(0)}$ are nested open and invariant
subsets
and $J$ is a purely non-dynamical ideal in $C^*_r(G|_{V\setminus U})$ with full support.

\begin{theorem}\label{thm:all ideals}
  Let $G$ be a locally compact Hausdorff \'etale groupoid that is inner-exact.
  There is a bijection $\Theta$ from $\T(G)$ to the collection of ideals of $C^*_r(G)$ such that
  \[
    \Theta(U,V,J) = \pi_U^{-1}(\iota_{V\setminus U}(J))
  \]
  for all $(U,V,J)\in \T(G)$.
  The inverse $\Theta^{-1}$ takes $I \lhd C^*_r(G)$ to the triple $(U,V,J)\in \T(G)$
  consisting of the sandwich sets $U\subset V$ and the purely non-dynamical ideal $J \lhd C^*_r(G|_{V\setminus U})$ with
  full support
  of~\cref{lem:sandwich}.
\end{theorem}

\begin{remark}\label{rmk:Theta->dynamical}
It is important in the statement of \cref{thm:all ideals} that $\varnothing$ is a groupoid, that its reduced $C^*$-algebra is $\{0\}$, 
and that $\{0\}$ is a purely non-dynamical ideal of $C^*_r(G)$: 
the dynamical ideals of $C^*_r(G)$ are in the range of $\Theta$ because each $I_U = \Theta(U, U, \{0\})$.
\end{remark}

\begin{proof}[Proof of Theorem~\ref{thm:all ideals}]
  The map $\Theta$ takes values in the ideals of $C^*_r(G)$ by definition.

  To see that $\Theta$ is injective, fix $(U, V, J) \in \mathcal{T}(G)$ and let $I = \Theta(U, V, J)$.
  We will prove that $U$ and $V$ are the sandwiching sets $U_I, V_I$ obtained from~\cref{lem:sandwich} applied to $I$,
  and that $J = \iota_{V\setminus U}^{-1}(\pi_U(I))$.
  This defines a left inverse to $\Theta$, defined on the image of $\Theta$, which implies that $\Theta$ is injective.

  We have $I_U = \pi_U^{-1}(0) \subset \Theta(U,V,J)$ by definition of $\Theta$.
  If $U' \subset G^{(0)}$ is an open invariant set containing $U$ such that $I_{U'} \subset \Theta(U,V,J)$,
  then $\pi_U(I_{U'}) \subset \pi_U(\Theta(U,V,J)) = \iota_{V\setminus U}(J)$ and the latter has trivial intersection
  with $G^{(0)}\setminus U$ (since $J$ is purely non-dynamical).
  As $\pi_U$ implements restriction of functions, we see that $I_{U'}\cap C_0(G^{(0)}) \subset C_0(U)$, so $U' = U$.
  Let $E$ be the faithful conditional expectation of $C^*_r(G)$ onto $C_0(G^{(0)})$.
  Observe that $E(\Theta(U,V,J)) \subset C_0(V)$.
  By~\cite[Lemma 3.6]{Bonicke-Li2017}, $\Theta(U,V,J)$ is contained in the ideal generated by $E(\Theta(U,V,J))$, so we
  see that $\Theta(U,V,J) \subset I_V$.
  In particular, $\supp(\Theta(U,V,J)) \subset \supp(I_V) = G|_V$.
  On the other hand, since $\supp(J) = G|_{V\setminus U}$ we have $G|_V \subset \supp(\Theta(U,V,J))$.
  Now if $V'$ is a proper open invariant subset of $V$ such that $\Theta(U,V,J) \subset I_{V'}$, then $G|_V =
  \supp(\Theta(U,V,J)) = \supp(I_{V'}) = G|_{V'} \subsetneq G|_V$
  which contradicts our observation above.
  Therefore, $V$ is the smallest such open invariant subset.
  Finally, observe that
    \begin{equation*}
      \iota_{V\setminus U}^{-1}(\pi_U(I)) = \iota_{V\setminus U}^{-1}(\pi_U(\Theta(U, V, J))) = \iota_{V\setminus
      U}^{-1}(\iota_{V\setminus U}(J)) = J,
    \end{equation*}
  and this completes the proof that $\Theta$ is injective.

  To see that it is surjective, fix an ideal $I$ of $C^*_r(G)$.
  By~\cref{lem:sandwich}, there are open invariant sets $U \subseteq V \subseteq G^{(0)}$ such that $I_U \subseteq I
  \subseteq I_V$ and
  $\supp(\pi^V_U(I/I_U)) = G|_{V \setminus U}$.
  Since $I \subseteq I_V = \iota_V(C^*(G|_V))$, we obtain an ideal $\iota_V^{-1}(I)$ of $C^*(G|_V)$.
  Let $J \coloneqq \pi^V_U(\iota_V^{-1}(I))$.
  We claim that $K \coloneqq \Theta(U, V, J)$ is equal to $I$, which will establish surjectivity of $\Theta$.
  By definition, both $I$ and $K$ are ideals of $C^*_r(G)$ that contain $I_U$, so it suffices to show that $I/I_U =
  K/I_U$.
  By inner-exactness, $\pi_U \colon C^*_r(G) \to C^*_r(G|_{G^{(0)}\setminus U})$ has kernel $I_U$, so it suffices to
  show  that $\pi_U(K) = \pi_U(\Theta(U, V, J))$.
  By definition of $\Theta$, we have $\pi_U(K) = \iota_{V \setminus U}(J) = \iota_{V \setminus
  U}(\pi^V_U(\iota_V^{-1}(I)))$.
  By definition of the two maps, $\iota_{V \setminus U} \circ \pi^V_U = \pi_U \circ \iota_V$, so we obtain $\pi_U(K) =
  \pi_U(I)$ as required.
\end{proof}

To link~\cref{lem:sandwich} back to the results~\cite[Lemma~2.6]{Hong-Szymanski}
and~\cite[Lemma~4.5]{anHuef-Raeburn1997} that inspired it,
we observe that for Deaconu--Renault groupoids, the dynamical ideals employed above are precisely the gauge-invariant
ideals of the $C^*$-algebra of a Deaconu--Renault groupoid.
The result is certainly well-known, but we are not aware that it has been recorded explicitly elsewhere in this
generality.
For the case of finitely aligned higher-rank graphs, this was observed in~\cite[Lemma 7.5]{Li2021}.

Recall that if $T \colon \N^d \curvearrowright X$ is an action by $d$ local homeomorphisms,
then we let $G_T$ denote the Deaconu--Renault groupoid of $T$ as in, for example, \cite[Section~3]{Sims-Williams}. An ideal $I$ of $C^*(G_T)$ is \emph{gauge-invariant} if the canonical gauge action $\gamma$ of $\TT^d$ on $C^*_r(G_T)$
satisfies $\gamma_z(I) \subseteq I$ for all $z \in \TT^d$.

\begin{proposition}\label{prop:gauge-inv ideals}
  Let $X$ be a locally compact Hausdorff space and suppose $T \colon \N^d \curvearrowright X$ is an action on $X$ by
  $d$ commuting local homeomorphisms.
  The map that carries an open invariant subset $U$ of $X$ to the ideal $I_U$ generated by $C_0(U)$ is a
  lattice isomorphism
  from the lattice of open invariant subsets of $X$ to the lattice of gauge-invariant ideals of $C^*(G_T)$.
\end{proposition}

\begin{proof}
  Since $\gamma_z(f) = f$ for all $z \in \TT^k$ and $f \in C_0(G^{(0)})$, the ideals of $C^*(G_T)$ generated by subsets of $C_0(G^{(0)})$ are gauge invariant. In particular, each $I_U$ is gauge-invariant.

  The map $U \mapsto I_U$ is an injection \cite[Theorem~10.3.3]{SimsCRM}.
  For surjectivity, we follow the second paragraph of the proof of~\cite[Theorem~10.3.3]{SimsCRM},
  dropping the assumption that $G$ is strongly effective but fixing a \emph{gauge-invariant} ideal $I$, until its
  penultimate sentence.
  At that point, while $GW$ need not be effective, we observe that $GW$ is identical to the groupoid of the topological
  higher-rank graph $\Lambda$ defined by
  $\Lambda^n = X \times \{n\}$ for all $n$, whose range and source maps are given by $s(x,n) = (T^n(x), 0)$ and $r(x,
  n) = (x , 0)$
  and with the factorisation rules $(x,m)(T^m(x), n) = (x, m+n) = (x, n)(T^n(x), m)$.
  We may now apply the gauge-invariant uniqueness theorem of~\cite[Corollary~5.21]{CarLarSimVit} in place
  of~\cite[Theorem~10.3.3]{SimsCRM} to see that $\tilde{\pi}$ is injective,
  and the surjectivity of $U \mapsto I_U$ follows.
  The final statement follows from~\cref{prop:dynamical ideals}.
\end{proof}

\section{Effectiveness at a unit and the obstruction ideal}\label{sec:obstruction ideal}

In this section we introduce the notions of effectiveness at a unit and joint
effectiveness at a unit for \'etale groupoids. The key property that emerges
is that of being jointly effective where effective. This is inspired by the
notions in~\cite[Section 7]{Ara-Lolk2018} of (strong) topological freeness at
a point for a partial group action. The points in the unit space of a
groupoid that are not effective comprise an open invariant set and hence a
dynamical ideal that we call the obstruction ideal. Our main results in this
section (\cref{thm:obstruction-ideal,cor:minimal}) say that if a Hausdorff
\'etale groupoid is inner-exact and its full and reduced $C^*$-algebras
coincide (Anantharaman-Delaroche calls this the weak containment property
\cite{A-Delaroche}), then the obstruction ideal contains all purely
non-dynamical ideals, and is minimal with this property.

Recall that a groupoid $G$ is \emph{effective} if the interior $\I^\circ(G)$
of the isotropy is equal to the unit space $G^{(0)}$. For $x \in G^{(0)}$, we
write $\I^\circ(G)_x$ for the intersection of $G_x$ with $\I^\circ(G)$.

\begin{definition} \label{def:effective-at-a-unit}
  A locally compact Hausdorff \'etale groupoid $G$ is \emph{effective at a unit $x\in G^{(0)}$} if $\I^\circ(G)_x = \{x\}$.
  Equivalently, $G$ is effective at $x$ if for any nontrivial isotropy element $\gamma\in \I(G)_x\setminus \{x\}$
  and any open bisection $B$ in $G\setminus G^{(0)}$ containing $\gamma$
  there exists $y\in s(B)$ such that $r(By) \neq y$.
  When the groupoid is understood, we may just say that the unit is effective.
  We let $G^{(0)}_{\eff}$ denote the collection of effective units.
\end{definition}

Any unit with trivial isotropy is effective. An isolated unit with nontrivial
isotropy is not effective.

We have the following general description of the units that are not effective.
This also shows that our terminology is consistent with the literature on effective groupoids.

\begin{lemma}\label{lem:effective-points}
  Let $G$ be a locally compact Hausdorff \'etale groupoid. We have
  \begin{equation}\label{eq:iso eqn}
    G^{(0)}\setminus G^{(0)}_{\eff} = s({\I^\circ(G)} \setminus G^{(0)}),
  \end{equation}
  and this is an open and invariant subset of $G^{(0)}$.
  Consequently, $G^{(0)}_{\eff}$ is closed and invariant.
  Moreover, $G$ is effective if and only if $G$ is effective at each of its units.
\end{lemma}
\begin{proof}
  Suppose that $G$ is not effective at $x \in G^{(0)}$.
  Then $x$ has nontrivial isotropy and any $\gamma\in \I^\circ(G)_x\setminus \{x\}$ is contained in an open bisection
  $B$ in $\I^\circ(G) \setminus G^{(0)}$.
  Therefore, $s(B)$ is an open neighbourhood of $x$ consisting of points that are not effective,
  so $G^{(0)}\setminus G^{(0)}_{\eff}$ is open and contained in $s(\I^\circ(G)\setminus G^{(0)})$.
  On the other hand, if $\gamma\in \I^\circ(G) \setminus G^{(0)}$, then there is an open bisection $B$ in $\I^\circ(G)
  \setminus G^{(0)}$ containing $\gamma$.
  If $x = s(\gamma)$, this means that $\I^\circ(G)_x \neq \{x\}$, so $G$ is not effective at $x$.

  In order to see invariance, let $x\in s(\I^\circ(G) \setminus G^{(0)})$ and take $\gamma\in G$ with $x = s(\gamma)$
  and $r(\gamma) = z \neq x$.
  We will show that $z$ is not effective.
  Let $\eta\in \I^\circ(G) \setminus G^{(0)}$ with $s(\eta) = x = r(\eta)$.
  Choose an open bisection $B_\gamma$ in $G\setminus G^{(0)}$ containing $\gamma$ and an open bisection $B_\eta$ in
  $\I(G)^\circ\setminus G^{(0)}$ containing $\eta$.
  Then $B_\gamma B_\eta B_\gamma^{-1}$ is an open bisection containing $\gamma \eta \gamma^{-1}$ (which is isotropy
  over $z$),
  and it consists only of isotropy elements, because $B_\eta$ consists only of isotropy elements.
  Therefore, $B_\gamma B_\eta B_\gamma^{-1} \subset \I(G)^{\circ} \setminus G^{(0)}$ and $z\in s(B_\gamma B_\eta
  B_\gamma^{-1})$,
  so $z$ is not effective.

  The final statement is a direct consequence of~\labelcref{eq:iso eqn}.
\end{proof}

The obstruction ideal defined below will play a central role in~\cref{thm:obstruction-ideal}.

\begin{definition}
Let $G$ be a locally compact Hausdorff \'etale groupoid.
The set of all units that are not effective is an open and invariant subset of $G^{(0)}$,
so it determines a dynamical ideal $I_{G^{(0)}\setminus G^{(0)}_{\eff}}$ of $C^*_r(G)$.
We call this the \emph{obstruction ideal} and denote it by $J^{\ob}$.
This terminology is explained in~\cref{rem:diagonal-uniqueness-theorem}.
\end{definition}

We let $G_{\eff}$ denote the reduction of $G$ to the closed invariant subset
of effective points. The unit space of $G_{\eff}$ then coincides with
$G^{(0)}_{\eff}$.

We require a groupoid analogue of the notion of strong topological freeness
introduced in~\cite[Section 7]{Ara-Lolk2018}.

\begin{definition}
  A locally compact Hausdorff groupoid $G$ is \emph{jointly effective} at a unit $x\in G^{(0)}$
  if for any finite collection of nontrivial isotropy elements $\gamma_1,\dots,\gamma_n\in \I(G)_x\setminus\{x\}$
  and any open bisections $B_1,\dots,B_n$ in $G\setminus G^{(0)}$ such that $\gamma_i \in B_i$
  there exists $y\in \bigcap_{i=1}^n s(B_i)$ such that $r(B_i y) \neq y$ for all $i=1,\dots,n$.
\end{definition}

\begin{remark}\label{rmk:effective->jointly effective}
If $G$ is effective, then it is jointly effective at every unit.
More generally, any unit in an open set of effective points is jointly effective.

For the first assertion, first suppose that $G$ is effective, and fix $x \in
G^{(0)}$ and $\gamma_1, \dots, \gamma_n \in \I(G)_x \setminus \{x\}$. Fix
open bisections $B_i$ in $G \setminus G^{(0)}$ containing $\gamma_i$. By
shrinking if necessary, we can assume that $W \coloneqq s(B_i) = s(B_j)$ for
all $i,j$. Since $G$ is effective, each $B_i \cap \I(G)$ has empty interior.
So for each $i$, the set $W_i \coloneqq s(B_i \setminus \I(G))$ is open and
dense in $W$. Hence $\bigcap_i W_i$ is open and dense, and in particular
nonempty. Now any $y \in \bigcap_i W_i$ satisfies $r(B_i y) \not= y$ for all
$i$.

For the second assertion, suppose only that $U$ is an open subset of
$G^{(0)}$ contained in $G^{(0)}_{\eff}$, and fix $x \in U$. Since
$G^{(0)}_{\eff}$ is invariant, $V \coloneqq r(GU)$ is open and invariant with $U
\subseteq V \subseteq G^{(0)}_{\textrm{eff}}$. The first assertion applied to
$G|_V$ shows that $x$ is jointly effective in $G|_V$, and hence in $G$.

It is possible for a groupoid to be effective at a unit but not jointly
effective at that unit---see~\cref{eg:ExelsX}.
\end{remark}

This leads us to an analogue of Ara and Lolk's notion of relative strong topological freeness.

\begin{definition}
Let $G$ be a locally compact Hausdorff \'etale groupoid. We say that $G$ is \emph{jointly effective where it is
effective} if $G$ is jointly effective at every point in $G^{(0)}_{\eff}$.
\end{definition}

\begin{examples}
\begin{enumerate}
\item By~\cref{rmk:effective->jointly effective}, if $G$ is effective then it
    is jointly effective where it is effective. In particular, if $G$ is
    principal, then it is jointly effective where it is effective.
  \item Suppose $G$ is a Hausdorff \'etale group bundle (for example, $G$ is a nontrivial discrete group). 
    Since $G^{(0)}$ is clopen, $G$ is effective at $x \in G^{(0)}$ if and only if $G_x = \{x\}$. 
    Since $G$ is trivially effective at $x$ when $G^x_x = \{x\}$, it follows that $G$ is jointly effective where it is effective. 
    We have $G^{(0)}_{\eff} = \{x : G_x = \{x\}\}$, and the obstruction ideal is generated by $C_0(\{x : G_x \not= \{x\}\})$.
\item In particular, Willett's groupoid \cite{Willett} consists entirely of
    isotropy, and hence is jointly effective where it is effective. It is
    not inner-exact. The obstruction ideal is the whole reduced groupoid
    $C^*$-algebra.
\end{enumerate}
\end{examples}

The next examples show that groupoids need not be jointly effective where
they are effective and that the property of being jointly effective where
effective does not necessarily pass to reductions to closed invariant
subsets. This latter permanence property does hold in groupoids all of whose
nontrivial isotropy groups are infinite cyclic (see~\cref{sec:DR-groupoids}).

\begin{example}[Exel's cross]\label{eg:ExelsX}
  Let $X =  \big([-1,1]\times \{0\}\big) \cup \big(\{0\}\times [-1,1]\big)$
  and consider the two homeomorphisms $\varphi$ and $\psi$ on $X$ given by
  $\varphi(x,y) = (-x,y)$ and $\psi(x,y) = (x,-y)$ for all $(x,y)\in X$.
  These commuting order-two homeomorphisms define an action
  $\varphi\oplus \psi\colon \Z/2\oplus \Z/2\Z \curvearrowright X$. Let
  $G_{\varphi\oplus \psi}$ be the transformation groupoid $X \rtimes
  (\Z/2\Z)^2$. To keep notation from getting too confusing, we regard
  $(\Z/2\Z)^2$ as the abelian group with four elements $\{e, a, b, ab\}$ (so
  the group operation is written multiplicatively), so that $a = (1,0)$
  and $b = (0,1)$ are the order-two generators.

  In this example, the interior of the isotropy $\I^\circ(G_{\varphi\oplus \psi})$ is
  \[
    %\I^\circ(G) =
    \big(X \times \{e\}\big) \cup \big((([-1,0)\cup(0,1]) \times \{0\}) \times \{b\}\big) \cup \big((\{0\} \times
    ([-1,0)\cup(0,1])) \times \{a\}\big),
  \]
  and the only effective unit is $(0,0) \in X$.

  Every point in $X$ has nontrivial isotropy (so $G_{\varphi\oplus \psi}$ is not effective).
  More specifically, the isotropy group of every point that is not the origin is isomorphic to $\Z/2\Z$ while the
  isotropy group at the origin is isomorphic to $\Z/2\Z \oplus \Z/2\Z$.
  The origin is the only point that is effective, but it is not jointly effective.
  Therefore, $G_{\varphi\oplus \psi}$ is not jointly effective where it is effective.
\end{example}

Ara and Lolk~\cite[Section 7]{Ara-Lolk2018} exhibit an example of a partial action
that shows that their relative strong topological freeness is not automatic,
and their example can be adapted to our groupoid setting.

\begin{example}
We can extend Exel's cross to see that being jointly effective where
effective does not pass to closed invariant subgroupoids. To see this, let
$X$ be as in Exel's cross, and let $Y = X \times [-1,1]$.

Extend $\varphi$ and $\psi$ to homeomorphisms $\tilde\varphi$ and $\tilde\psi$
on $Y$ by $\tilde{\varphi}(x, t) = (\varphi(x),-t)$ and similarly
$\tilde{\psi}(x,t) = (\psi(x), -t)$. Again we regard these as determining an
action of $\Z_2 \oplus \Z_2 = \{e, a, b, ab\}$ on $Y$.

Neither $a$ nor $b$ fixes any point in $Y \setminus X$ because both invert
the $t$-coordinate. Since the only point in $X$ fixed by $ab$ is the point
$(0,0) \in X$, the only points in $Y$ fixed by $ab$ are those of the form
$((0,0), t)$. So $G_{\tilde{\varphi}, \tilde{\psi}} \coloneqq Y \rtimes
(\Z/2\Z)^2$ is effective, and in particular jointly effective where it is
effective. However, its reduction to the closed invariant set $X$ is Exel's
cross, which is not jointly effective where it is effective.
\end{example}

The next lemma is an easy adaptation of~\cite[Lemma 29.4]{Exel2017} from
partial actions of groups to groupoids, so we give just a fairly succinct
proof.

\begin{lemma}\label{lem:Exel's_lemma}
  Let $G$ be a Hausdorff \'etale groupoid, let $x\in G^{(0)}$ be a unit, and let $B$ be an open bisection
  such that $B \cap \I(G)_x = \varnothing$. Let $f\in C_c(G)$ be such that $f$ has support in $B$.
  Given $\varepsilon > 0$ there exists $h\in C_0(G^{(0)})$ satisfying $0\leq h\leq 1$, $h$ is constantly $1$ on a neighbourhood of $x$,
  and $\| h f h \| < \varepsilon$.
\end{lemma}

\begin{proof}
First suppose that $x \not\in s(B)$. By Urysohn's lemma we can find $h \in
C_0(G^{(0)}, [0,1])$ that is $1$ on a neighbourhood of $x$ and vanishes on
$\{s(\gamma) : |f(\gamma)| \ge \varepsilon\} \subseteq s(B)$. Since $f$ is
supported on a bisection, its $C^*$-norm agrees with its supremum norm \cite[Corollary 9.3.4]{SimsCRM},
and hence $\|hfh\|_{C^*(G)} = \|hfh\|_\infty < \varepsilon$.

Now suppose that $x \in s(B)$. Let $\gamma$ be the unique element of $B$ with
$s(\gamma) = x$. By assumption, $r(\gamma) \not= x$ so we can choose an open
set $V_1$ containing $x$ such that $r(V_1) \cap V_1 = \varnothing$. By
Urysohn's lemma, there exists $h \in C_0(G^{(0)}, [0,1])$ such that $h = 1$
on a neighbourhood of $x$ and $h$ vanishes off $V_1$. In particular,
$\supp(h) \cap r(B\supp(h)) = \varnothing$, and so $hfh = 0$.
\end{proof}

The next two results say that when an inner-exact groupoid $G$ whose full and
reduced $C^*$-algebras coincide is jointly effective where it is effective,
its obstruction ideal $J^{\ob}$ is the minimal dynamical ideal that contains
all purely non-dynamical ideals of $C^*_r(G)$. The proof of the first result
closely follows that of~\cite[Theorem 7.12]{Ara-Lolk2018} (which does not
require the weak containment property) with only minor modifications.

\begin{remark}
The hypothesis below that the sequence $0 \to J^{\ob} \to C^*_r(G) \to C^*_r(G_{\eff}) \to 0$ is exact holds
if, for example, $G$ is inner-exact (in particular, if it is amenable).
However, it also holds trivially if $G$ is effective, and we invoke it in that situation in~\cref{prp:BCFS extension}.
So we have stated~\cref{thm:obstruction-ideal} accordingly.
\end{remark}

\begin{theorem}\label{thm:obstruction-ideal}
  Let $G$ be a locally compact Hausdorff \'etale groupoid that is jointly effective where it is effective.
  Let $J^{\ob}$ be the obstruction ideal in $C^*_r(G)$ and suppose the sequence $0 \to J^{\ob} \to C^*_r(G) \to C^*_r(G_{\eff}) \to
  0$ is exact. If $I$ is a purely non-dynamical ideal of $C^*_r(G)$
  %so that $I \cap C_0(G^{(0)}) = \{0\}$,
  then $I \subset J^{\ob}$.
\end{theorem}

\begin{proof}
  We suppose that $I \not\subseteq J^{\ob}$ and derive a contradiction. Fix $a\in
  I \setminus J^{\ob}$.
  In particular, $a^*a\in I\setminus J^{\ob}$.
  Let $E\colon C^*_r(G) \to C_0(G^{(0)})$ be the canonical faithful conditional expectation,
  let $G_{\eff} = G|_{G^{(0)}_{\eff}}$,
  and let $\pi = \pi_{G^{(0)}\setminus G^{(0)}_{\eff}}\colon C^*_r(G) \to C^*_r(G_{\eff})$ denote the
  canonical quotient map. Let $E_{\eff}$ be the canonical faithful conditional expectation associated with $C^*_r(G_{\eff})$.
  Then the diagram
% In the tikzcd below:
% quotes are activated - note that ' places label in opposite location.
  \[
  \begin{tikzcd}
    C^*_r(G) \arrow[r, "\pi"] \arrow[d, "E"'] & C^*_r(G_{\eff}) \arrow[d, "E_{\eff}"] \\
    C_0(G^{(0)}) \arrow[r, "\pi"'] & C_0(G^{(0)}_{\eff})
  \end{tikzcd}
  \]
  commutes. By hypothesis, $J^{\ob} = \ker(\pi)$, so $\pi(a^*a) \neq 0$ because $a^*a\notin J^{\ob}$.
  Since $E_{\eff}$ is faithful, $E_{\eff}(\pi(a^*a)) \neq 0$. Hence $0 \neq E_{\eff}(\pi(a^*a))
  = \pi(E(a^*a))$
  by commutativity of the diagram.
  This means that $f\coloneqq E(a^*a)\in C_0(G^{(0)})$ is nonzero on $G^{(0)}_{\eff}$.
  Choose $x_0\in G^{(0)}_{\eff}$ such that
  \begin{equation}
    |f(x_0)| = \sup_{x\in G^{(0)}_{\eff}} |f(x)|,
  \end{equation}
  and let $0 < \varepsilon < \frac{|f(x_0)|}{2}$.
  The set $V = \{ x\in G^{(0)} : |f(x)| < |f(x_0)| + \varepsilon/4 \}$ is open in $G^{(0)}$ and contains $x_0$.
  By Urysohn's lemma we may pick a function $u\in C_0(G^{(0)})$ such that $0\leq u \leq 1$, $u(x_0) = 1$, and $u$
  vanishes outside $V$.
  Set $z \coloneqq u a^*a\in I\setminus J^{\ob}$ and observe that $E(z) = u E(a^*a) = u f$, and
  \begin{equation}\label{eq:||E(z)||}
    2 \varepsilon < |f(x_0)| \leq \| E(z) \| \leq |f(x_0)| + \varepsilon/4.
  \end{equation}

  We claim that there exists $h\in C_0(G^{(0)})$ satisfying $0 \leq h \leq 1$, $h(x_1) = 1$ and
  \begin{align}
    \|E(z)\| &< \|h E(z) h\| + \varepsilon; \label{claim:1} \\
    \|h E(z) h - h z h\| &< \varepsilon. \label{claim:2}
  \end{align}
  Since $z\in C^*_r(G)$, there exists $g\in C_c(G)$ such that $\|z - g\| < \varepsilon/4$.
  In particular, $\|E(z) - E(g)\| < \varepsilon/4$.
  Note that $E(g)$ is supported on $G^{(0)}$ and $g - E(g)\in C_c(G\setminus G^{(0)})$.
  Choose open bisections $B_1,\ldots,B_k\subset G \setminus G^{(0)}$ that cover $\supp(g - E(g))$
  and write
  \begin{equation}\label{eq:decomposition}
    g - E(g) = \sum_{i = 1}^k g_i
  \end{equation}
  with $g_i\in C_0(B_i)$ for each $i$.

  For each $i$ such that $B_i \cap \I(G)_{x_0} = \varnothing$,
  we can apply~\cref{lem:Exel's_lemma} to obtain a function $h_i\in C_0(G^{(0)}, [0,1])$ that is identically~$1$ on an open neighbourhood $U_i$ of $x_0$
  and satisfies $\|h_i g_i h_i\| \le \varepsilon/2k$.
  Consider the open neighbourhood $U \coloneqq \{x \in G^{(0)} : |E(g)(x) - E(g)(x_0)|  < \varepsilon/4\}$ of $x_0$.
  Since $G$ is jointly effective at $x_0$ by hypothesis,
  there exists a unit $x_1 \in U \cap \bigcap_{B_i \cap \I(G)_{x_0} = \varnothing} U_i$
  such that for each $i$ satisfying $B_i \cap \I(G)_{x_0} \not= \varnothing$, we have $r(B_i x_1) \not=  x_1$.
  Since $x_1 \in U$, we have
  \begin{equation}\label{eq:x_1}
    |E(g)(x_1) - E(g)(x_0)| < \varepsilon/4,
  \end{equation}
  and for each $i$ such that $B_i \cap \I(G)_{x_0} = \varnothing$, since $x_1 \in U_i$ we have $h_i(x_1) = 1$.

  For each $i$ such that $B_i \cap \I(G)_{x_0} \not= \varnothing$, \cref{lem:Exel's_lemma} for $B_i$
  at $x_1$ yields a function $h_i \in C_0(G^{(0)}, [0,1])$ satisfying
  \begin{equation}\label{eq:hgh}
    h_i(x_1) = 1 \qquad\text{and}\qquad \|h_i g_i h_i \| < \frac{\varepsilon}{2 k}.
  \end{equation}
  Altogether we have constructed functions $h_1, \dots, h_k$ that all satisfy~\labelcref{eq:hgh}.
  Set $h \coloneqq \prod^k_{i=1} h_i \in C_0(G^{(0)})$ and note that $0\leq h \leq 1$ and $h(x_1) = 1$.

  It remains to verify~\labelcref{claim:1,claim:2}; we do this by direct computation.
  Using~\labelcref{eq:||E(z)||} and the fact that $u(x_0) = 1$, we see that
  \[
    \| E(z) \| - \varepsilon \leq |f(x_0)| - 3\varepsilon/4 = |E(z)(x_0)| - 3\varepsilon/4.
  \]
  By first using the choice of $g$ and then the choice of $x_1$ from~\labelcref{eq:x_1}, we find
  \[
    |E(z)(x_0)| - 3\varepsilon/4 < |E(g)(x_0)| - \varepsilon/2 < |E(g)(x_1)| - \varepsilon/4.
  \]
  Remembering that $h(x_1) = 1$, we obtain
  \[
    |E(g)(x_1)| - \varepsilon/4 = |(h E(g) h)(x_1)| -\varepsilon/4 \leq \|h E(g) h\| - \varepsilon/4 < \|h E(z) h\|.
  \]
  This means that $\| E(z) \| - \varepsilon < \| h E(z) h \| $ so~\labelcref{claim:1} follows.
  For~\labelcref{claim:2}, we use the decomposition~\eqref{eq:decomposition} and then~\eqref{eq:hgh} to see that
  \[
    \| h g h - h E(g) h \| = \Big\| \sum_{i = 1}^k h g_i h \Big\| \leq \sum_{i = 1}^k \|h g_i h\| < \varepsilon/2.
  \]
  Hence
  \[
    \|h z h - h E(z) h \| \leq \|hzh - hgh\| + \|hgh - hE(g)h\| + \|hE(g)h - hE(z)h\| < \varepsilon,
  \]
  and this proves~\labelcref{claim:2}.

  To complete the proof, consider the canonical quotient map $q\colon C^*_r(G) \to C^*_r(G)/ I$ which is injective on the diagonal, since
  $I\cap C_0(G^{(0)}) = \{0\}$ by hypothesis.
  Since $z\in I$ we have $q(h E(z) h) = q(h E(z) h - h z h)$ so
  \[
    \| h E(z) h \| = \|q(h E(z) h) \| = \| q(h E(z) h - h z h)\| \leq \| h E(z) h - h z h\|.
  \]
  Applying~\labelcref{claim:1}, the above inequality, and then~\labelcref{claim:2}, we
  obtain
  \[
    \|E(z)\| < \|h E(z) h - h z h\| + \varepsilon < 2\varepsilon,
  \]
  This contradicts the estimate $2 \varepsilon < \|E(z)\|$ from~\labelcref{eq:||E(z)||}.
  Hence $I \subset J^{\ob}$.
\end{proof}

The lemma below uses the full groupoid $C^*$-algebra $C^*(G)$. We refer the
reader to~\cite{WilliamsGroupoidsBook} for a discussion of this $C^*$-algebra
that does not assume second-countability.

\begin{lemma} \label{lem:BCFS}
  Let $G$ be a locally compact Hausdorff \'etale groupoid whose full and reduced $C^*$-algebras coincide.
  Let $J^{\ob}$ be the obstruction ideal in $C^*_r(G)$.
  There is a $^*$-representation $\varepsilon$ of the full groupoid $C^*$-algebra $C^*(G)$ such
  that $\ker(\varepsilon)$ is purely non-dynamical and such that $\supp(J^{\ob}) \subset \supp(\ker(\varepsilon))$.
\end{lemma}
\begin{proof}
  The proof of~\cite[Proposition~5.2]{Brown-Clark-Farthing-Sims}
  shows that for each $x \in G^{(0)}$ there is an $I$-norm bounded $^*$-representation
  $\varepsilon_x$ of $C_c(G)$ on the orbit space $\ell^2([x])$ such that $\varepsilon_x(f)e_y = \sum_{\gamma \in G_y} f(\gamma) e_{r(\gamma)}$.
  By definition of $C^*(G)$, $\varepsilon_x$ extends to a representation of $C^*(G)$.
  Let $\varepsilon = \bigoplus_{x \in G^{(0)}} \varepsilon_x$  (this representation is also
  described on \cite[page~330]{KwasniewskiMeyer2019}). Then $\varepsilon$ is injective on $C_0(G^{(0)})$,
  because for $f \in C_0(G^{(0)})$ and $x \in G^{(0)}$, we have $0 \not= f(x) = \big(\varepsilon_x(f) e_x \mid e_x\big) \le \|\varepsilon(f)\|$.
  Hence $\ker(\varepsilon)$ is a purely non-dynamical ideal.

  To see that $\supp(\ker(\varepsilon))$ contains $\supp(J^{\ob})$,
  by~\cref{lem:support invariant} it suffices to show that $G^{(0)}\setminus G^{(0)}_{\eff} = \supp(J^{\ob} \cap G^{(0)}) \subseteq \supp(\ker(\varepsilon))$.
  Fix $x \in G^{(0)}\setminus G^{(0)}_{\eff}$ and choose $\gamma\in \I^\circ(G) \setminus G^{(0)}$ such that $s(\gamma) = x$.
  Take an open bisection neighbourhood $B \subseteq \I^\circ(G)\setminus G^{(0)}$ of~$\gamma$.
  Choose a nonzero function $f\in C_c(s(B))$ with $f(x) \neq 0$ and let $\tilde{f}\in C_c(B)$ be the function given by $\tilde{f}(\eta) = f(s(\eta))$ for all $\eta\in B$.
  By extending by zero, both functions can be regarded as elements of $C_c(G)$. Direct calculation on basis elements
  (see the proof of~\cite[Proposition~5.5(2)]{Brown-Clark-Farthing-Sims}) shows that $f - \tilde{f} \in \ker(\varepsilon)$.
  So $x \in \supp(\ker(\varepsilon))$.
\end{proof}

\begin{corollary}\label{cor:minimal}
  Let $G$ be a locally compact Hausdorff \'etale groupoid that is jointly effective where it is effective.
  Suppose the sequence $0\to J^{\ob} \to C^*_r(G) \to C^*_r(G_{\eff})\to 0$ is exact
  and that the full and reduced groupoid $C^*$-algebras of $G$ coincide.
  Then there is a purely non-dynamical ideal whose support is equal to that of
  $J^{\ob}$, and $J^{\ob}$ is the minimal dynamical ideal that contains all purely non-dynamical ideals of $C^*_r(G)$.
\end{corollary}
\begin{proof}
\Cref{lem:BCFS} gives a purely non-dynamical ideal $I$ such that
$\supp(J^{\ob}) \subseteq \supp(I)$. \Cref{thm:obstruction-ideal} shows that
$J^{\ob}$ contains all purely non-dynamical ideals in $C^*_r(G)$, and in
particular contains $I$. Hence $\supp(I) \subseteq \supp(J^{\ob})$, and we
obtain equality. Now suppose that $I_U$ is a dynamical ideal that contains
every purely non-dynamical ideal. Then in particular, $I \subseteq I_U$.
Hence $(G^{(0)} \setminus G^{(0)}_{\eff}) \subseteq \supp(J^{\ob}) = \supp(U)
\subseteq \supp(I_U) = G_U$. Thus \cref{prop:dynamical ideals} implies that
$J^{\ob} \subseteq I_U$.
\end{proof}

To finish the section, we observe that our results can be used to recover
\cite[Proposition~5.5(2)]{Brown-Clark-Farthing-Sims}, without the assumption
that $G$ is second-countable. This is not new. For example, it can be
recovered from a special case of \cite[Theorem~7.29]{KwasniewskiMeyer2019}.
We include it here only to illustrate how our results relate to effective
groupoids.

\begin{proposition}\label{prp:BCFS extension}
Let $G$ be a locally compact Hausdorff \'etale groupoid.
\begin{enumerate}[label=(\arabic*)]
  \item \label{i:if_effective} If $G$ is effective, then every nontrivial ideal of $C^*_r(G)$
    contains a nonzero element of $C_0(G^{(0)})$.
  \item \label{i:then_effective} If every nontrivial ideal of the full $C^*$-algebra $C^*(G)$ contains a
    nonzero element of $C_0(G^{(0)})$, then the full and reduced $C^*$-algebras of $G$ coincide and $G$ is effective.
\end{enumerate}
\end{proposition}

\begin{proof}
\labelcref{i:if_effective} Fix an ideal $I$ of $C^*_r(G)$ that contains no
nonzero element of $C_0(G^{(0)})$; we must show that $I = \{0\}$. Since $G$
is effective, it is jointly effective where it is effective, and $J^{\ob}$ is
trivial. The sequence $0 \to J^{\ob} \to C^*_r(G) \to C^*_r(G_{\eff}) \to 0$
is then trivially exact, and~\cref{thm:obstruction-ideal} implies that $I
\subseteq J^{\ob} = \{0\}$.

\labelcref{i:then_effective} We prove the contrapositive. First suppose that
the full and reduced $C^*$-algebras of $G$ do not coincide. Then the kernel
of the regular representation $\lambda : C^*(G) \to C^*_r(G)$ is a nonzero
purely non-dynamical ideal. Now suppose that the full and reduced
$C^*$-algebras of $G$ coincide but that $G$ is not effective. Then $J^{\ob}$
is nontrivial, and Lemma~\ref{lem:BCFS} implies that there is a purely
non-dynamical ideal $I$ of $C^*_r(G)$ whose support contains that of
$J^{\ob}$, and in particular is nonzero.
\end{proof}

\begin{remark}\label{rem:diagonal-uniqueness-theorem}
A mainstay of the theory of \'etale groupoid $C^*$-algebras is the \emph{diagonal
uniqueness theorem}, dating back to \cite{Renault}: for amenable effective
\'etale groupoids, any $^*$-homomorphism that is injective on the diagonal is
injective (See~\cref{prp:BCFS extension}). If $G$ is a groupoid that does not
satisfy the conclusion of this theorem, then there is a $^*$-homomorphism
$\phi$ of $C^*_r(G)$ whose kernel is purely non-dynamical. So if $G$ is also
inner-exact Hausdorff \'etale groupoid whose full and reduced $C^*$-algebras
coincide, then the kernel of $\phi$ is contained in the obstruction ideal.
This justifies the terminology \emph{obstruction ideal}: the obstruction
ideal measures how far away a groupoid is from satisfying a diagonal
uniqueness theorem.

For example, if $G$ is the groupoid of a higher-rank graph in the sense of
\cite{KumjianPask2000}, then the obstruction ideal is zero if and only if the
higher-rank graph is aperiodic (so its $C^*$-algebra satisfies the
Cuntz--Krieger uniqueness theorem) \cite[Proposition 3.6]{Robertson-Sims}.
\end{remark}

\section{Examples}\label{sec:examples}

\subsection{Groupoids from local homeomorphisms}\label{sec:DR-groupoids}
First we consider the groupoid constructed from a local homeomorphisms $T$ on
a locally compact Hausdorff space $X$. The associated semi-direct product
groupoid, usually called the Deaconu--Renault groupoid, is
\[
  G_T = \bigcup_{m,n\in \N} \{ (x, m-n, y )\in X\times \{m-n\}\times X: T^m x = T^n y\}
\]
where the product of $(x,p,y)$ and $(y',q,z)$ is defined precisely if $y=y'$ in which case $(x,p,y)(y,q,z) = (x,p+q,z)$
while inversion is $(x,p,y)^{-1} = (y, -p, x)$.
The unit space is naturally identified with $X$ and the range and source maps are then $r(x,p,y) = x$ and $s(x,p,y) =
y$.

We first verify that this groupoid is jointly effective where it is
effective.  For
open subsets $U$ and $V$ of $X$, the sets of the form
\[
  Z(U,m,n,V) = \{ (x,m-n,y)\in G_T : x\in U, y\in V\}
\]
comprise a basis for a locally compact Hausdorff \'etale topology on $G_T$.
The groupoid $G_T$ is amenable, and hence inner-exact \cite[Section~3]{Sims-Williams}.

For the rank-one Deaconu--Renault groupoids, we can describe explicitly the
points that are not effective . For $p\in \N_+$, let
\begin{equation}
  \mathcal{P}_p  = \{ x\in X : T^p\text{ pointwise fixes a neighbourhood of }x\}
\end{equation}
and let $\mathcal{P} = \bigcup_{p=1}^\infty \mathcal{P}_p$. Then
$\mathcal{P}$ is open and invariant in $X$ and the restricted system
$(\mathcal{P}, T)$ is reversible.

\begin{lemma}
  For a local homeomorphism $T$ on a locally compact Hausdorff space $X$, we have
  \begin{equation}\label{eq:noneff}
    X\setminus X_{\eff} = \{ x\in X : \orb_T(x) \cap \mathcal{P} \neq \varnothing \}.
  \end{equation}
\end{lemma}

\begin{proof}
  Let $V \coloneqq \{ x\in X : \orb_T(x) \cap \mathcal{P} \neq \varnothing \}$
  and let $G = G_T$ be the Deaconu--Renault groupoid of $T$.
  It is straightforward to verify that $V$ is open and invariant in $X$.
  We verify~\labelcref{eq:noneff} one inclusion at a time.

  Let $x\in V$ and choose $l\in \N$ such that $x' \coloneqq T^l(x)\in \mathcal{P}_p$ for some $p\in \N_+$.
  Pick an open set $U\subset X$ all of whose points are $p$-periodic
  and consider the open bisection given by
  \begin{equation}
    B = \{ (y, p, y)\in G : y\in U\}.
  \end{equation}
  Note that $B\subset \I^\circ(G) \setminus X$.
  In particular, $\I^\circ(G)_{x'}$ contains $(x',p,x')$, so $x'$ is not effective, and by invariance $x$ is not
  effective.

  For the other inclusion, suppose $x$ is not effective.
  Then $(x,p,x)\in \I^\circ(G)_x$ for some $p\in \N_+$. Fix an open bisection $B$ in $\I^\circ(G)\setminus X$
  containing $(x,p,x)$.
  We may assume that $T^p x = x$, so by shrinking $B$ we may assume that $B \subseteq Z(U, p,
  0, U)$ for some open subset $U$ of $X$.
  Then $T^p y = y$ for every $y\in s(B)$ since $B \subseteq \I(G)$.
  Therefore, $x\in V$.
\end{proof}

Next we show that any Deaconu--Renault groupoid $G_T$ is jointly effective where it is effective.
The result actually only depends on the nontrivial isotropy being infinite cyclic, so we record this more general
result here.

\begin{lemma}
  Any Hausdorff \'etale groupoid $G$ whose nontrivial isotropy is infinite cyclic is jointly effective where it is
  effective.
\end{lemma}

\begin{proof}
  Let $x\in G^{(0)}$ be a point with nontrivial isotropy and suppose $B_1,\dots,B_N$ are open bisections in $G$ such
  that
  each $B_i$ contains an element $\gamma_i\in \Iso(G)_x\setminus G^{(0)}$.
  Since the isotropy group at $x$ is infinite cyclic there are minimal integers $p_1,\dots,p_N$ such that
  $\gamma_i^{p_i} = \gamma_j^{p_j}$ for all $i,j=1,\dots,N$.
  Put $\gamma \coloneqq \gamma_i^{p_i}$.
  Then $B \coloneqq B_1^{p_1}\cap \dots\cap B_N^{p_N}$ is an open bisection containing $\gamma$.

  Assume now that $x$ is effective.
  So whenever $U\subset G^{(0)}$ is an open neighbourhood of $x$,
  there is a point $y\in U$ such that $r(B y) \neq y$.
  Applying this to a neighbourhood basis of $x$, we can find a sequence $(y_n)_n$ in $G^{(0)}$ such that $y_n \to x$
  and $r(By_n) \neq y_n$ for all $n$.
  We show that $G$ is jointly effective at $x$. It suffices to show that for large $n$, we have $r(B_i y_n) \neq y_n$ for all $i =
  1,\dots,N$.

  Since $B_1^{p_1}$ contains $\gamma$, we have $x\in s(B_1^{p_1})$ so $y_n\in s(B_1^{p_1})$ for large $n$,
  and since $B \subset B_1^{p_1}$ we see that $r(B_1^{p_1} y_n) = r(B y_n)$.
  If $r(B_1 y_n) = y_n$, then
  \[
    y_n = r(B_1 y_n) = r(B_1^{p_1} y_n) = r(B y_n),
  \]
  which contradicts our choice of $y_n$.
  So for large $n$ we have $r(B_1 y_n) \not= y_n$ as required.
  Hence $G$ is jointly effective at $x$.
\end{proof}

As an immediate corollary we see that the groupoids built from a local homeomorphism $T$ on a locally compact Hausdorff
space $X$, called rank-one Deaconu--Renault groupoids,
are covered by the above result.

\begin{corollary}
  Any rank-one Deaconu--Renault groupoid is jointly effective where it is effective.
\end{corollary}

\subsection{Partial actions}\label{sec:AraLolk}
Our notion of being jointly effective for groupoids is directly inspired by
Ara and Lolk's notion of relative strong topological freeness for partial
actions \cite[Section 7]{Ara-Lolk2018}. A partial action $\theta\colon
\Gamma \curvearrowright X$ of a countable discrete group $\Gamma$ on a locally compact
Hausdorff space $X$ is topologically free at $x\in X$ if whenever
$\theta_g(x) = x$ for some $1 \neq g\in \Gamma$, for any open neighbourhood
$U$ of $x$, there exists $y\in U$ such that $\theta_g(y) \neq y$. We say
$\theta$ is \emph{strongly topologically free} at $x$ if for any finite
collection $1\neq g_1,\cdots,g_k\in \Gamma$ such that $\theta_{g_i}(x) = x$
and any neighbourhood $U$ around $x$, there exists $y\in U$ such that
$\theta_{g_i}(y) \neq y$ for all $i=1,\cdots,k$. Finally, $\theta$ is
\emph{relatively strong topologically free} if it is strongly topologically
free at all points at which it is topologically free.

Following~\cite[Section 2]{Abadie}, a partial action $\theta\colon \Gamma\curvearrowright X$ has an associated groupoid
\[
   G_\theta = \{ (x, g, y)\in X\times \Gamma\times X \mid y\in \textrm{dom}(g),~\theta_g(y) = x\}
\]
whose unit space $G_\theta^{(0)}$ is naturally identified with $X$. Elements
$(x, g, y)$ and $(y', g', z)$ in $G_\theta$ are composable if and only if $y
= y'$ in which case $(x, g, y)(y', g', z) = (x, g g', z)$. Inversion is given
by ${(x, g, y)}^{-1} = (y, g^{-1}, x)$. The source and range maps $s, r\colon
G_\theta\to X$ are $s(x,g,y) = y$ and $r(x,g,y) = x$. The groupoid $G_\theta$
carries a locally compact and Hausdorff \'etale topology.

\begin{lemma}
  Let $\theta\colon \Gamma\curvearrowright X$ be a partial action of a countable discrete group $\Gamma$ on a locally
  compact Hausdorff space $X$.
  Then $\theta$ is topologically free at $x\in X$ if and only if $G_\theta$ is effective at $x$.
  Moreover, $\theta$ is strongly topologically free at $x$ if and
  only if $G_\theta$ is jointly effective at $x$.
  In particular, $\theta$ is relatively strongly topologically free if and only if $G_\theta$ is jointly effective
  where it is effective.
\end{lemma}
\begin{proof}
Suppose that $\theta$ is not strongly topologically free at $x$. There exist
$g_1, \dots, g_k \in \Gamma \setminus \{e\}$ that all fix $x$, and a
neighbourhood $U$ of $x$ such that for every $y \in U$ there exists $i$ such
that $\theta_{g_i}(y) = y$. For each $i$, define
\[
    B_i \coloneqq \{(\theta_{g_i}(y), g_i, y) : y \in U\}.
\]
Then each $B_i$ is a bisection containing $(x, g_i, x)$, and there is no $y
\in \bigcap_i s(B_i) = U$ such that $r(B_i y) \not= y$ for all $i$. So
$G_\theta$ is not jointly effective at $x$. Taking $k = 1$ shows that if
$\theta$ is not topologically free at $x$ then $G_\theta$ is not effective at
$x$.

Now suppose that $G_\theta$ is not jointly effective at $x$. Fix elements
$\gamma_1, \dots, \gamma_k \in \I(G_\theta)_x \setminus \{x\}$ and open
bisections $B_i$ containing $\gamma_i$ such that for each $y \in \bigcap_i s(B_i)$
there exists $i$ such that $r(B_i y) = y$. By definition of $G_\theta$, each
$\gamma_i = (x, g_i, x)$ for some $g_i \in \Gamma \setminus \{e\}$. By
definition of the topology on $G_\theta$, for each $i$ there is an open
neighborhood $U_i$ of $x$ such that $\{(\theta_{g_i}(y), g_i, y) : y \in
U_i\} \subseteq B_i$. Now $U = \bigcap_i U_i$ is a neighbourhood of $x$ and
for each $y \in U$ there exists $i$ such that $r(B_i y) = y$. That is,
$\theta_{g_i}(y) = y$. So $\theta$ is not strongly topologically free at $x$.
Again, taking $k = 1$ throughout shows that if $G_\theta$ is not effective at
$x$ then $\theta$ is not topologically free at $x$.

The final statement follows by definition.
\end{proof}


\begin{thebibliography}{99}
\bibitem[Ab04]{Abadie} F. Abadie,
\emph{On Partial Actions and Groupoids},
Proc. Amer. Math. Soc. \textbf{132} (2004), no. 4, 1037--1047.

\bibitem[A-D19]{A-Delaroche} C.~Anantharaman-Delaroche,
\emph{Some remarks about the weak containment property for groupoids and semigroups},
arXiv preprint (arXiv:1604.01724v4 [math.OA]).

\bibitem[A-D21]{Anatharaman-Delaroche2021} C.~Anantharaman-Delaroche,
\emph{Exact groupoids},
arXiv preprint (arXiv:1605.05117v2 [math.OA]).

\bibitem[A-DR00]{A-DR} C.~Anantharaman-Delaroche and J.~Renault,
\emph{Amenable groupoids},
With a foreword by Georges Skandalis and Appendix B by E. Germain,
L'Enseignement Math\'ematique, Geneva, 2000, 196.

\bibitem[AL18]{Ara-Lolk2018} P. Ara and M. Lolk,
\emph{Convex subshifts, separated Bratteli diagrams, and ideal structure of tame separated graph algebras},
Adv. Math. \textbf{328} (2018), 367--435.

\bibitem[BL17]{Bonicke-Li2017} C. B\"onicke and K. Li, \emph{Ideal structure
  and pure infiniteness of ample groupoid $C^*$-algebras}, Ergodic Theory
Dynam. Systems \textbf{40} (2020), 34--63.

\bibitem[BCS23]{BCS23} K.A. Brix, T.M. Carlsen, and A. Sims,
\emph{Ideal structure of $C^*$-algebras of commuting local homeomorphisms},
arXiv preprint (arXiv:2303.02313 [math.OA]).

\bibitem[BCFS]{Brown-Clark-Farthing-Sims} J.~Brown, L.O.~Clark, C.~Farthing,
and A.~Sims, \emph{Simplicity of algebras associated to \'etale
groupoids}, Semigroup Forum \textbf{88} (2014), 433--452.

\bibitem[CLSV11]{CarLarSimVit} T.M.~Carlsen, N.S.~Larsen, A.~Sims, and
S.T.~Vittadello, \emph{Co-universal algebras associated to product
systems, and gauge-invariant uniqueness theorems}, Proc. Lond. Math. Soc.
(3) \textbf{103} (2011), 563--600.

\bibitem[Ex17]{Exel2017} R.~Exel,
\emph{Partial dynamical systems, Fell bundles and applications},
Mathematical Surveys and Monographs, 224. American Mathematical Society, Providence, RI, 2017. vi+321 pp.
ISBN: 978--1--4704--3785--5.

\bibitem[EP2017]{ExelPardo} R.~Exel and E.~Pardo,
\emph{Self-similar graphs, a unified treatment of {K}atsura and {N}ekrashevych {$C^*$}-algebras},
Adv. Math. \textbf{306} (2017), 1046--1129.

\bibitem[HS04]{Hong-Szymanski} J.H.~Hong and W.~Szyma\'nski, \emph{The
primitive ideal space of the $C^*$-algebras of infinite graphs}, J. Math.
Soc. Japan \textbf{56} (2004), 45--64.

\bibitem[aHR97]{anHuef-Raeburn1997} A.~an Huef and I.~Raeburn, \emph{The
ideal structure of Cuntz--Krieger algebras}, Ergodic Theory Dynam.
Systems \textbf{17} (1997), 611--624.

\bibitem[Ka2008]{KatsuraModels} T.~Katsura,
\emph{A construction of actions on {K}irchberg algebras which induce given actions on their {$K$}-groups},
J. reine angew. Math. \textbf{617} (2008), 27--65.

\bibitem[Ka2021]{KatsuraIdeals} T.~Katsura,
\emph{Ideal structure of $C^*$-algebras of singly generated dynamical systems},
arXiv preprint (arXiv:2107.10422 [math.OA]).

\bibitem[Ku86]{Kumjian1986} A.~Kumjian, \emph{On $C^*$-diagonals}, Canad. J.
Math. \textbf{38} (1986), 969--1008.

\bibitem[KP00]{KumjianPask2000} A.~Kumjian and D.~Pask,
\emph{Higher rank graph $C^*$-algebras},
New York J. Math \textbf{6} (2000), 1--20.

\bibitem[KPRR97]{KPRR} A.~Kumjian, D.~Pask, I.~Raeburn and J.N.~Renault,
\emph{Graphs, groupoids, and {C}untz--{K}rieger algebras},
J. Funct. Anal. \textbf{144} (1997), 505--541.

\bibitem[KM19]{KwasniewskiMeyer2019} B. Kwasniewski and R. Meyer, \emph{Essential Crossed Products for
Inverse Semigroup Actions: Simplicity and Pure Infiniteness}, Documenta Math. \textbf{26} (2021), 271--335.

\bibitem[Li21]{Li2021} X.~Li,
\emph{Left regular representations of Garside categories I. $C^*$-algebras and groupoids},
Glasgow Math. J. \textbf{65}, (2023), S53--S86.

\bibitem[Re80]{Renault} J.~Renault,
\emph{A groupoid approach to {$C^*$}-algebras},
Lecture Notes in Mathematics \textbf{793}, Springer, Berlin (1980), MR 584266.

\bibitem[Re91]{Renault1991} J.~Renault,
\emph{The ideal structure of groupoid crossed product {$C^*$}-algebras},
J. Operator Theory \textbf{25} (1991), 3--36.

\bibitem[Re08]{Renault2008} J.~Renault,
\emph{Cartan subalgebras in {$C^*$}-algebras},
Irish Math. Soc. Bull. \textbf{61} (2008), 29--63.

\bibitem[RS07]{Robertson-Sims} D. Robertson and A. Sims,
\emph{Simplicity of $C^*$-algebras associated to higher-rank graphs},
Bull. London Math. Soc. \textbf{39} (2007) 337--344.

\bibitem[Si20]{SimsCRM} A.~Sims,
\emph{Hausdorff étale groupoids and their $C^*$-algebras},
in Operator algebras and dynamics: groupoids, crossed
products and Rokhlin Dimension (F. Perera, Ed.) in \emph{Advanced Coursed
in Mathematics. CRM Barcelona}, Birkhäuser, 2020.

\bibitem[SW16]{Sims-Williams} A.~Sims and D.P.~Williams,
\emph{The primitive ideals of some \'etale groupoid {$C^*$}-algebras},
Algebras and Representation Theory \textbf{19} (2016), 255--276.

\bibitem[Sp07]{Spielberg} J.~Spielberg,
\emph{Graph-based models for {K}irchberg algebras},
J. Operator Theory \textbf{57} (2007), 347--374.

\bibitem[Wi15]{Willett} R.~Willett, \emph{A non-amenable groupoid whose
maximal and reduced $C^*$-algebras are the same}, M\"unster J. Math.
\textbf{8}, (2015), 241--252.

\bibitem[Wil19]{WilliamsGroupoidsBook} D.P.~Williams,
\emph{A tool kit for groupoid $C^*$-algebras},
\emph{Mathematical Surveys and Monographs} Volume 241, American Mathematical Society, Providence, RI (2019), 398 pp.
\end{thebibliography}
\end{document}